\documentclass[12pt]{article}

\usepackage{amssymb}   
\usepackage{amsfonts}  
\usepackage{amsmath}   

\usepackage{tikz}
\usetikzlibrary{calc}
\usetikzlibrary{decorations.pathreplacing}
\usetikzlibrary{shapes,snakes}

\usepackage{setspace}
\usepackage{subfigure}
\usepackage{url}
\usepackage{graphicx}
\usepackage{booktabs}
\usepackage{ifthen}
\usepackage{enumerate}
\usepackage{comment}

\usepackage[margin=1in]{geometry}

\newtheorem{thm}{Theorem}[section]
\newtheorem{lem}[thm]{Lemma}

\numberwithin{equation}{section}

\begin{document}

\title{A Solution to the Edge-Balanced Index Set Problem\\ for Complete Odd Bipartite Graphs\thanks{This paper has been accepted for publication in the Bulletin of the Institute of Combinatorics and Its Applications.}} 

\author{
Elliot Krop\thanks{Department of Mathematics, Clayton State University, (\texttt{ElliotKrop@clayton.edu})}
\and
Sarah Minion\thanks{Department of Mathematics, Clayton State University, (\texttt{sminion@student.clayton.edu})}
\and 
Pritul Patel\thanks{Department of Mathematics, Clayton State University, (\texttt{ppatel15@student.clayton.edu})}
\and
Christopher Raridan\thanks{Department of Mathematics, Clayton State University, (\texttt{ChristopherRaridan@clayton.edu})}}
   
\maketitle

\begin{abstract}
In 2009, Kong, Wang, and Lee began work on the problem of finding the edge-balanced index sets of complete bipartite graphs $K_{m,n}$ by solving the cases where $n=1$, $2$, $3$, $4$, and $5$, and also the case where $m=n$. In 2011, Krop and Sikes expanded upon that work by finding $EBI(K_{m,m-2a})$ for odd $m>5$ and $1 \leq a \leq \frac{m-3}{4}$. In this paper, we provide a general solution to the edge-balanced index set problem for all complete odd bipartite graphs, thereby concluding the problem for this case. 
\\[\baselineskip] 
\\[\baselineskip]
\end{abstract}

\section{Introduction}

\subsection{Definitions}

For a graph $G=(V,E)$ with vertex set $V$ and edge set $E$, a \textit{binary edge-labeling} is an injective function $f : E \to \{ 0,1 \}$. An edge labeled $1$ will be called a \textit{$1$-edge} and an edge labeled $0$ will be called a \textit{$0$-edge}. Let $e(1)$ and $e(0)$ represent the number of edges labeled $1$ and $0$, respectively. A binary edge-labeling is \textit{edge-friendly} if $|e(1)-e(0)| \leq 1$. Call the number of $1$-edges incident with a vertex $v$ the \textit{$1$-degree} of $v$, denoted $\deg_1(v)$, and the number of $0$-edges incident with $v$ the \textit{$0$-degree}, denoted $\deg_0(v)$, and note that the degree of $v$ is $\deg(v) = \deg_1(v) + \deg_0(v)$. An edge-friendly labeling of $G$ will induce a (possibly partial) \textit{vertex-labeling} where a vertex $v$ will be labeled $1$ when $\deg_1(v) > \deg_0(v)$, labeled $0$ when $\deg_0(v) > \deg_1(v)$, and will be unlabeled when $\deg_1(v) = \deg_0(v)$. A vertex labeled $1$ will be called a \textit{$1$-vertex} and a vertex labeled $0$ will be called a \textit{$0$-vertex}. Let $v(1)$ and $v(0)$ represent the number of $1$-vertices and $0$-vertices, respectively. The \textit{edge-balanced index set} of $G$ is defined as
\begin{align*}
EBI(G) = \big\{ |v(1)-v(0)|: \text{over all edge-friendly labelings of $G$} \big\}.
\end{align*} 
A reader interested in the study of graph labelings may find Gallian's dynamic survey~\cite{GallianYYYY} helpful as an introduction to graph labeling problems.

Let $K_{m,n}$ be a complete bipartite graph with parts of cardinality $m$ and $n$ where $m \geq n$ are positive odd integers. Any edge-friendly labeling of such a complete bipartite graph has $|e(1)-e(0)| = 1$ and since every vertex has odd degree, every vertex must be labeled either $0$ or $1$, so $v(1)+v(0)=m+n$.

\subsection{History and Motivation}

The idea of a balanced labelings was introduced in 1992 by Lee, Liu, and Tan~\cite{LLT1992}. In 1995, Kong and Lee provided results concerning edge-balanced graphs~\cite{KL1995}. In~\cite{KWL2009}, Kong, Wang, and Lee introduced the problem of finding the $EBI$ of complete bipartite graphs by solving the cases where $n=1$, $2$, $3$, $4$, and $5$, and also the case where $m=n$, but left the other cases open. In~\cite{KS2011}, Krop and Sikes expanded upon the work of Kong, Wang, and Lee by finding $EBI(K_{m,m-2a})$ for odd $m>5$ and $1 \leq a \leq \frac{m-3}{4}$.    

In this paper, we conclude the edge-balanced index set problem for complete bipartite graphs $K_{m,n}$ where $m \geq n$ are positive odd integers. In particular, we compute the maximal element of $EBI(K_{m,n})$ and provide an edge-friendly labeling that maximizes $|v(1)-v(0)|$. We then show that all smaller elements of the edge-balanced index set can be obtained.


\section{Finding the maximal element of $EBI(K_{m,n})$}


\begin{lem}
\label{lem:max-EBI}
Let $K_{m,n}$ be a complete bipartite graph with parts of cardinality $m$ and $n$, where $m \geq n$ are positive odd integers. Then the maximal element of $EBI(K_{m,n})$, denoted $\max EBI(K_{m,n})$, is
\begin{align}
\max EBI(K_{m,n})  = 
\begin{cases}
2, &\text{if}\; n=1, \\
m+n-2k-2, &\text{otherwise},
\end{cases}
\end{align}   
where $k = \left\lceil \frac{m-1}{n+1} \right\rceil$.  
\end{lem}

\begin{proof}

Let $K_{m,n}$ be a complete bipartite graph with parts $A$ and $B$ of cardinality $m \geq n$, respectively, where $m,n$ are positive odd integers. Consider an edge-friendly labeling of this graph such that $k$ vertices in part $A$ are labeled $0$ and $m-k$ vertices are labeled $1$, and $j$ vertices in part $B$ are labeled $0$ and $n-j$ vertices are labeled $1$. Without loss of generality, we may assume that $v(1) \geq v(0)$. Since all vertices are labeled either $0$ or $1$, the maximal element of $EBI(K_{m,n})$ is achieved when $k$ and $j$ are minimized, and in this case, $\max EBI(K_{m,n}) = v(1) - v(0) = (m-k)+(n-j)-(k+j) = m+n-2(k+j)$. 


Since $K_{m,n}$ has $mn$ edges, and $mn$ is odd, any edge-friendly labeling will have either $e(0) = \frac{mn-1}{2}$ and $e(1) = \frac{mn+1}{2}$ or else $e(0) = \frac{mn+1}{2}$ and $e(1) = \frac{mn-1}{2}$; without loss of generality, we choose an edge-friendly labeling that has the former so that $e(1) - e(0) = 1$. In order to minimize $k$, we want the $0$-vertices in part $A$ to be incident with $0$-edges only; that is, we force the $1$-edges to be incident with the other $m-k$ vertices in this part. The number of $1$-edges, $e(1) = \frac{mn+1}{2}$, should be divided among these $m-k$ vertices, which means the average $1$-degree will be $\frac{e(1)}{m-k}$. For each of these $m-k$ vertices to be labeled $1$, this average $1$-degree must be greater than or equal to $\frac{n+1}{2}$. Then $k \geq \frac{m-1}{n+1}$, so we let $k = \left\lceil \frac{m-1}{n+1} \right\rceil$, which is greater than or equal to $1$ unless $m=n=1$, in which case $k=0$. Similarly, we want $n-j$ vertices in part $B$ to be labeled $1$, so the average $1$-degree, $\frac{e(1)}{n-j}$, must be greater than or equal to $\frac{m+1}{2}$. Then $j \geq \frac{n-1}{m+1}$, so we let $j = \left\lceil \frac{n-1}{m+1} \right\rceil$. Note that $j=0$ if and only if $n=1$ and $j = 1$ in all other cases. 


Now consider a labeling where $k < \left\lceil \frac{m-1}{n+1} \right\rceil$, say $k = \left\lceil \frac{m-1}{n+1} \right\rceil - a$ for some positive integer $a$. Counting the $0$- and $1$-degrees of vertices in part $A$, we find that $e(0) \leq kn +(m-k)\frac{n-1}{2}$ and $e(1)\geq (m-k)\frac{n+1}{2}$, so that
\begin{align}
e(1)-e(0) &\geq m - \left( \left\lceil \frac{m-1}{n+1} \right\rceil -a \right) - \left( \left\lceil \frac{m-1}{n+1} \right\rceil - a \right) n \nonumber \\
	&= m - (n+1)\left\lceil \frac{m-1}{n+1} \right\rceil + a(n+1) \label{eqn:bad-k}
\end{align}
Using the division algorithm, we write $m-1 = (n+1)q+r$ where $0 \leq r < n+1$.
If $r=0$, then $(n+1) \left\lceil \frac{m-1}{n+1} \right\rceil = m-1$, and~\eqref{eqn:bad-k} implies that $e(1)-e(0) \geq 1 + a(n+1) > 1$, since $a$ and $n$ are positive integers. Similarly, if $0 < r < n+1$, then $(n+1) \left\lceil \frac{m-1}{n+1} \right\rceil = m+n-r$, and~\eqref{eqn:bad-k} implies $e(1)-e(0) \geq a+r+n(a-1) > 1$. In either case, the labeling is not edge-friendly. A similar argument applies to a labeling where $j$ is chosen as $0$ instead of $1$ when $n \geq 3$.


Therefore, if $m$ is a positive odd integer and $n=1$, then $m-1$ is even, $k = \left\lceil \frac{m-1}{n+1} \right\rceil = \frac{m-1}{2}$, and $j=0$, so that the maximal element of $EBI(K_{m,1})$ is $m+n-2(k+j) = 2$. Moreover, for odd integers $m \geq n \geq 3$, we have that $j=1$ and $\max EBI(K_{m,n}) = m+n-2(k+1)$, where $k = \left\lceil \frac{m-1}{n+1} \right\rceil$. \hfill $\Box$
\end{proof}

Note that $\max EBI(K_{m,n})$ is an even integer for positive odd integers $m \geq n$. 

\section{The $\max EBI(K_{m,n}) $ labeling}
\label{sec:max-EBI}


Consider a complete bipartite graph $K_{m,n}$ with parts $A$ and $B$ of order $m $ and $n$, respectively, where $m \geq n$ are positive odd integers. Let $k = \left\lceil \frac{m-1}{n+1} \right\rceil$ and $j = \left\lceil \frac{n-1}{m+1} \right\rceil$. Name the vertices in part $A$ as $\{ v_1, \dots, v_{m} \}$ and those in part $B$ as $\{ u_1, \dots, u_{n}\}$. The goal is to provide an edge-friendly labeling such that the vertices $\{ v_1, \dots, v_{k} \}$ and $u_1$ will be $0$-vertices while $\{ v_{k+1}, \dots, v_{m} \}$ and $\{ u_2, \dots, u_{n} \}$ will be $1$-vertices.


In the case where $n=1$, we have $k = \frac{m-1}{2}$ and $j=0$, so for $1 \leq i \leq \frac{m-1}{2}$, we label the edge $u_1v_i$ by~$0$, and for $\frac{m+1}{2} \leq i \leq m$, we label edge $u_1v_i$ by~$1$. Then $e(0) = \frac{m-1}{2}$, $e(1) = \frac{m+1}{2}$, and the labeling is edge-friendly since $e(1)-e(0) = 1$. Moreover, because $\deg_1(u_1) = \frac{m+1}{2}$ implies $u_1$ is a $1$-vertex along with vertices $v_i$ for $\frac{m+1}{2} \leq i \leq m$, we have that $v(0) = e(0)$ and $v(1) = e(1)+1$, which implies $v(1)-v(0) = 2$.


For odd $n \geq 3$, we have that $k = \left\lceil \frac{m-1}{n+1} \right\rceil$ and $j=1$. For each integer $2 \leq i \leq n$ and for each integer $1 \leq i' \leq \frac{m+1}{2}$, we label edge $u_iv_{s(i,i')}$ by $1$, where 
\begin{align*}
s(i,i') = \left( \left[ (i-2)\frac{m+1}{2}+ i'-1 \right] \!\!\!\!\!\! \mod{(m-k)} \right) + k+1.
\end{align*}
This function counts through the integers $\{ k+1, k+2, \dots m \}$ consecutively with wraparound, distributing $1$-edges as uniformly as possible among the vertices $\{ v_{k+1}, \dots, v_{m} \}$ in part $A$ and $\{ u_2, \dots, u_{n} \}$ in part $B$. At this point, we have labeled $(n-1)\frac{m+1}{2}$ edges by $1$. Since $e(1) = \frac{mn+1}{2}$, there are still $e(1) - (n-1)\frac{m+1}{2} = \frac{m-n+2}{2} \geq 1$ edges that need to be labeled~$1$ to obtain an edge-friendly labeling. So, for each integer $1 \leq i' \leq \frac{m-n+2}{2}$, we label edge $u_1v_{s(1,i')}$ by~$1$, where 
\begin{align*}
s(1,i') = \left( \left[ s \left( n, \frac{m+1}{2} \right) - k + i' - 1 \right] \!\!\!\!\!\! \mod{(m-k)} \right) + k+1.
\end{align*}
Label all the remaining edges in $K_{m,n}$ by $0$. 


Counting the number of $1$-edges incident with each vertex $v_i$, where $k+1 \leq i \leq m$, we find that $\deg_1(v_i) \geq \frac{n+1}{2}$, which implies that each vertex in $\{ v_{k+1}, \dots, v_m \}$ is a $1$-vertex. Since $\deg_1(u_i) = \frac{m+1}{2}$, where $2 \leq i \leq n$, each vertex in $\{ u_2, \dots, u_{n} \}$ is a $1$-vertex as well. Since $\deg_1(u_1) = \frac{m-n+2}{2} < \frac{m+1}{2}$, vertex $u_1$ will be a $0$-vertex. Likewise, each of the vertices $\{ v_1, \dots, v_k \}$ is incident with $0$-edges only, so these vertices are $0$-vertices. Thus, for this edge-labeling, we have the following counts (by construction): $e(1) = \frac{m-n+2}{2} + (n-1)\frac{m+1}{2} = \frac{mn+1}{2}$, $e(0) = mn - e(1) = \frac{mn-1}{2}$, $v(1) = (m-k) + (n-1) = m+n-(k+1)$, and $v(0) = k+1$.  Then $e(1) - e(0) = 1$ and the labeling is edge-friendly. Moreover, under this labeling, $v(1)-v(0) = m+n-2(k+1) \in EBI(K_{m,n})$, and by Lemma~\ref{lem:max-EBI}, $\max EBI(K_{m,n}) = m+n-2(k+1)$. Thus, we have an edge-friendly labeling of $K_{m,n}$ for odd integers $m \geq n \geq 3$ that attains the maximal value in $EBI(K_{m,n})$.

\section{Smaller indices}


The following observation, taken from~\cite{KWL2009}, is helpful in determining elements in the edge-balanced index set for graphs that have all vertices of odd degree: If $G$ is a graph whose vertices all have odd degree, then the edge-balanced index set of $G$ contains only even integers. Thus, when searching for terms smaller than $\max EBI(K_{m,n})$ in the edge-balanced index set for complete bipartite graphs with both parts of odd order, we need only produce (show there exists) an edge-friendly labeling for which the quantity $|v(1) - v(0)|$ is even and falls between $0$ and $\max EBI(K_{m,n}) - 2$, inclusive, if the quantity is realizable.


\begin{thm}
\label{thm:EBI}
Let $K_{m,n}$ be a complete bipartite graph with parts of cardinality $m$ and $n$, where $m \geq n$ are positive odd integers. Then
\begin{align}
EBI(K_{m,n}) = 
\begin{cases}
\{ 2 \}, &\text{if}\; n=1, \\
\{ 0, 2, \dots, \max EBI(K_{m,n}) \}, &\text{otherwise},
\end{cases}
\end{align}   
where $\max EBI(K_{m,n})$ is given by Lemma~\ref{lem:max-EBI}. 
\end{thm}

\begin{proof}
We consider an edge-friendly labeling of such a complete bipartite graph $K_{m,n}$ as given the proof of Lemma~\ref{lem:max-EBI}, labeled so that $v(1)-v(0)=\max EBI(K_{m,n})$ as described in Section~\ref{sec:max-EBI}. If $n=1$, then every edge-friendly labeling of $K_{m,1}$ has $|v(1)-v(0)| = \max EBI(K_{m,n}) = 2$, so $EBI(K_{m,1}) = \{ 2 \}$.

For odd $n \geq 3$, we assume $v(1) > v(0)$; otherwise, $v(1)=v(0)$ and we have that $0 \in EBI(K_{m,n})$. The edge-friendly condition requires that parts $A$ and $B$ both contain $0$-vertices and $1$-vertices. Without loss of generality, choose vertices $x,y \in A$, where $x$ is a $0$-vertex and $y$ is a $1$-vertex, and note that for any such choice, we have $\deg_0(x),\deg_1(y) \geq \frac{n+1}{2}$. By the pigeonhole principle, vertices $x$ and $y$ are adjacent to at least one common neighbor $z$ by $0$-edge $xz$ and $1$-edge $yz$. Switching the labels on $xz$ and $yz$ either preserves or decreases by $2$ the quantity $v(1)-v(0)$. Moreover, edge-friendliness implies that 
\begin{align}
\sum_{v \in V} \deg_1(v) - \sum_{v \in V} \deg_0(v) = 2. \label{eqn:e-f}
\end{align}

We claim that for all such choices of vertices $x,y \in A$, $\deg_0(x) > \deg_1(y)$. Assume the contrary: for any $0$-vertex $x$ and $1$-vertex $y$ in part $A$, suppose that 
\begin{align}
\deg_0(x) \leq \deg_1(y). \label{eqn:star}
\end{align}
Let $V_0$ and $V_1$ represent the sets of $0$- and $1$-vertices, respectively. Then inequality~\eqref{eqn:star} implies that
\begin{align}
\sum_{v \in V_1} \deg_1(v) - \sum_{v \in V_0} \deg_0(v) \geq \big( v(1) - v(0) \big) \frac{n+1}{2}. \label{eqn:1}
\end{align} 
Note that the right-hand side of~\eqref{eqn:1} is the product of the number of additional $1$-vertices (since $v(1)>v(0)$) and the minimum of their $1$-degrees. Since 
\begin{align}
\sum_{v \in V_1} \deg(v) - \sum_{v \in V_0} \deg(v) = \big( v(1) - v(0) \big) n, \label{eqn:2}
\end{align}  
from~\eqref{eqn:1} it follows that
\begin{align}
\sum_{v \in V_1} \deg_0(v) - \sum_{v \in V_0} \deg_1(v) \leq \big( v(1) - v(0) \big) \frac{n-1}{2}. \label{eqn:3}
\end{align} 
Expanding the sums in~\eqref{eqn:e-f} gives the following:
\begin{align}
2 &= \sum_{v \in V_1} \deg_1(v) + \sum_{v \in V_0} \deg_1(v) - \sum_{v \in V_1} \deg_0(v) - \sum_{v \in V_0} \deg_0(v) \nonumber \\
	&\geq \sum_{v \in V_1} \deg_1(v) - \sum_{v \in V_0} \deg_0(v) + \big( v(1) - v(0) \big) \frac{n-1}{2}, \label{eq:rearrange}
\end{align}
by~\eqref{eqn:3}. Rearranging~\eqref{eq:rearrange} gives
\begin{align*}
\sum_{v \in V_1} \deg_1(v) - \sum_{v \in V_0} \deg_0(v) &\leq 2 - \big( v(1) - v(0) \big) \frac{n-1}{2},
\end{align*}
which implies by~\eqref{eqn:1} that
\begin{align*}
\big( v(1) - v(0) \big) \frac{n+1}{2} \leq 2 - \big( v(1) - v(0) \big) \frac{n-1}{2},
\end{align*}
or equivalently, $\big( v(1) - v(0) \big) n \leq 2$. This is a contradiction since $v(1)>v(0)$ and $n \geq 3$. Thus, for all $0$-vertices $x$ and $1$-vertices $y$ in $A$, $\deg_0(x) > \deg_1(y)$.

We conclude that if we switch the labels on $0$-edge $xz$ and $1$-edge $yz$, where $z$ is a common neighbor to $x$ and $y$, the label on vertex $z$ will not change, but vertex $y$ will change from being a $1$-vertex to a $0$-vertex before vertex $x$ changes from being a $0$-vertex to a $1$-vertex. At the moment vertex $y$ changes its label from $1$ to $0$, the quantity $v(1)-v(0)$ is decreased by $2$. We may repeat the process on the new edge-friendly labeling, continually decreasing the edge-balanced index by $2$, until $v(1)=v(0)$, at which time we are done. \hfill $\Box$   
\end{proof}



\begin{thebibliography}{99}
	\bibitem{GallianYYYY} J.A.~Gallian, A dynamic survey of graph labeling, \textit{Electron. J. Combin. 5}, \textbf{DS6} (2012), \texttt{http://www.combinatorics.org/ojs/index.php/eljc/index}.
 
	\bibitem{KL1995} M.C.~Kong and S.-M.~Lee, On edge-balanced graphs, \textit{Graph Theory, Comb. and Alg.}, \textbf{1} (1995), 711-722.
 
 	\bibitem{KWL2009} M.C.~Kong, Y.-C.~Wang, and S.-M.~Lee, On edge-balanced index sets of some complete $k$-partite graphs, In \textit{Proceedings of the Fortieth Southeastern International Conference on Combinatorics, Graph Theory and Computing}, \textbf{196} (2009), pp. 71-94.

	\bibitem{KS2011} E.~Krop and K.~Sikes, On the edge-balanced index sets of complete bipartite graphs, In \textit{Proceedings of the Forty-Second Southeastern International Conference on Combinatorics, Graph Theory and Computing}, \textbf{207} (2011), pp. 23-32.
 
	\bibitem{LLT1992} S.-M.~Lee, A.~Liu, and S.K.~Tan, On balanced graphs, \textit{Cong. Numer.}, \textbf{87} (1992), pp. 59-64. 
\end{thebibliography}
\end{document}